\documentclass[12pt,letterpaper]{amsart}
 \usepackage[leqno]{amsmath}
 \usepackage{amsfonts}
 \usepackage{amssymb}
 \usepackage{amsthm} 
 \usepackage{amssymb}
 \usepackage[all]{xy}
 \usepackage{graphicx}
 \usepackage{xcolor}
 \usepackage{hyperref}
 \usepackage{lipsum}
 \usepackage{lineno}
 \usepackage[normalem]{ulem}
 \usepackage{tikz-cd} 
 \usepackage{multicol}
 \usepackage[numbers]{natbib}
 
 \usepackage{scalerel,stackengine}
 \stackMath
 \newcommand\reallywidehat[1]{%
 	\savestack{\tmpbox}{\stretchto{%
 			\scaleto{%
 				\scalerel*[\widthof{\ensuremath{#1}}]{\kern-.6pt\bigwedge\kern-.6pt}%
 				{\rule[-\textheight/2]{1ex}{\textheight}}
 			}{\textheight}%
 		}{0.5ex}}%
 	\stackon[1pt]{#1}{\tmpbox}%
 }
 \newcommand{\dbZ}{{\mathbb Z}}

 \newcommand{\la}{\langle}
 \newcommand{\ra}{\rangle}

 \newcommand{\ad}{\mbox{ad}}
 \newcommand{\slnr}{\mathfrak{sl}_n(R)} 
 \newcommand{\slfr}{\mathfrak{sl}_4(R)}
 \newcommand{\sltr}{\mathfrak{sl}_3(R)}
 \newcommand{\stnr}{\mathfrak{st}_n(R)}
 \newcommand{\stfr}{\mathfrak{st}_4(R)}
  \newcommand{\sttr}{\mathfrak{st}_3(R)}

     \newcommand{\slnrbar}{\mathfrak{sl}_n(\mathfrak{R})}

 \theoremstyle{plain}
 \newtheorem{theorem}{Theorem}[section]
 \newtheorem*{theorem*}{Theorem}
 \newtheorem*{claim*}{\textit{Claim}}
  \newtheorem{claim}{\textit{Claim}}
 \newtheorem{corollary}[theorem]{Corollary}
 \newtheorem{prop}[theorem]{Proposition}
 \newtheorem{lemma}[theorem]{Lemma}

 \theoremstyle{definition}
 \newtheorem{definition}[theorem]{Definition}
 \newtheorem{notation}[theorem]{Notation}
 
 \newtheorem*{remarkstar}{\textit{Remark}}
 \newtheorem{remark}[theorem]{Remark}

 \numberwithin{equation}{section}

\begin{document}

\title{Finite presentability of  universal central extensions of ${\mathfrak{sl}_n}$ }
\author{Efim Zelmanov}
\address{Department of Mathematics, University of California at San Diego,
9500 Gilman Drive, La Jolla, California 92093-0112, USA}
\author{Zezhou Zhang}
\address{Department of Mathematics, Beijing Normal University, China, Beijing 100875}
\email{zz2d@bnu.edu.cn}
\thanks{* The second-named  author is partially supported by ``the Fundamental Research Funds for the Central Universities'', Grant No. 2018NTST15.}

\begin{abstract}
	In this paper we discuss finite presentability of the universal central extensions of Lie algebras  ${\mathfrak{sl}_n(R)}$, where $n\geq 3$ and $R$ is a unital associative $k$-algebra. We show that a universal central extension is finitely presented if and only if the algebra $R$ is finitely presented.
\end{abstract}

\renewcommand\thesubsection{\textit{\thesection}.\Roman{subsection}}
\renewcommand\thesubsubsection{\textit{\thesubsection}.\Alph{subsubsection}}
\maketitle

\section{Introduction}

Recall that if $\mathfrak{g}$ is a Lie algebra and $\mathfrak{a}$ an abelian Lie algebra,  then $f: \mathcal{L} \rightarrow \mathfrak{g} $ is called a central extension of $\mathfrak{g}$ (by $\mathfrak{a}$) if it fits into the exact sequence \[0\rightarrow \mathfrak{a} \rightarrow \mathcal{L} \stackrel{f}{\rightarrow} \mathfrak{g} \rightarrow 0,\] where $\ker f \in Z(\mathcal{L})$, the center of $\mathcal{L}$. A central extension $u: \mathcal{L}\rightarrow \mathfrak{g}$ is called a universal central extension if there exists a unique homomorphism from $u: \mathcal{L}\rightarrow \mathfrak{g}$ to any other central extension $u: \mathcal{M}\rightarrow \mathfrak{g}$ of $\mathfrak{g}$. When $\mathfrak{g}$ is a perfect Lie algebra, the existence of its universal central extension is guaranteed. This extension is perfect as well.   Study of universal central extensions of perfect groups and Lie algebras goes back to the work of Schur (see \cite{schur1904darstellung}). See \cite{Neher} for record of fundamental results, along with references.

In this paper, we focus on $\widehat{\mathfrak{sl}_n(R)}$, the universal central extension of $\mathfrak{sl}_n(R)$, where $R$ is a finitely generated $k$-algebra, and $k$ a commutative ring. Recall that $\slnr$ is the Lie algebra generated by off-diagonal matrix units among $n \times n$ matrices; it is a subalgebra of $\mathfrak{gl}_n(R)$. Equivalently,  $\slnr=\{M\in \mathfrak{gl}_n(R) \mid \mbox{tr} (M) \in [R,R] \}$.  Another common notation for this Lie algebra is $\mathfrak{e}_n(R)$ (where e stands for ``elementary''), as its precise analog among groups is denoted $E_n(R)$, the  elementary linear group.  

When $R$ is a free $k$-module, $\widehat{\mathfrak{sl}_n(R)}$ for $n\geq3$ have been studied in \cite{BlochLiealgebra,Kassellodaycentral,GaoShang}. In the central extension \[0\rightarrow H_2(\mathfrak{sl}_n(R),k) \rightarrow\widehat{\mathfrak{sl}_n(R)}  \rightarrow \slnr \rightarrow 0,\] the extension part  $H_2(\mathfrak{sl}_n(R),k)$ is isomorphic to Connes' cyclic homology group $HC_1 (R)$ when $n\geq 5$, while having $HC_1 (R)$ as a direct summand for $n=3,4$ (see \cite{GaoShang}). This has been a motivation for computing such extensions (see \cite{Kassellodaycentral,loday2013cyclic}).

Is $\widehat{\mathfrak{sl}_n(R)}$ finitely presented as a $k$-Lie algebra? We remark that for Lie algebras $\mathfrak{g}$, it is well known that (see \cite[7.5.2]{WeibelHA}) infinite dimensionality of $H_2(\mathfrak{g},k)$ implies that $\mathfrak{g}$ is non-finitely presented. This method allows us to derive negative answers: for example, in the setting of the previous paragraph,  $\mathfrak{sl}_n(R)$ is not finitely presented if $HC_1 (R)$ is infinite dimensional. However, it gives nothing for $\widehat{\mathfrak{sl}_n(R)}$, whose second homology is zero.

 The analogue of this question, phrased  for Steinberg groups $\text{St}_n(R)$, was considered by Rehmann-Soule in \cite{RehmannSoule} for $R$ commutative and by Kristic-McCool in \cite{KrsticMcCool} for general $R$. The connection between $\text{St}_n(R)$ and the universal central extension of $\mbox{E}_n(R)$ is well documented: for example, they are equal when $n\geq 4$ and $R$ commutative, while $\varinjlim\text{St}_n(R)$ is the universal central extension of $\varinjlim\text{E}_n(R)$, for all rings $R$(see \cite[1.4.13]{hahn1989classical}). 

Our main result is:

\begin{theorem}\label{mainintro}
	Let $k$ be a commutative associative ring, let $n\geq 3$ and let $R$ be a finitely generated unital associative $k$-algebra. Then $R$ being finitely presented as a $k$-algebra is equivalent to $ \widehat{\slnr}$ being  finitely presented as a Lie algebra over $k$.
\end{theorem}

The paper is organized as follows: Section \ref{sec conventions} provides several notational and conceptual conventions; Section \ref{sectionpres} offers detailed presentations of $\widehat{\mathfrak{sl}_n(R)}$, as well as certain technical preparations for proving the Theorem \ref{mainintro}; Section \ref{secfinpres} proves Theorem \ref{mainintro}, barring a crucial lemma;  the last three sections are devoted to the proof of this lemma.

\section{Conventions}\label{sec conventions}
This paper adheres to following conventions: all rings are unital and associative; $\ad(x)y$ stands for $[x,y]$; $k$ is always a commutative ring; all Lie algebras are perfect, guaranteeing the existence of their universal central extension.

\section{Presentations of $\widehat{\slnr}$}\label{sectionpres}

We take our first step toward theorem \ref{mainintro} by providing presentations for $\widehat{\slnr}$. These presentations use Steinberg Lie algebras $\stnr$ as a starting point . As a matter of fact, we tend to think of algebras $\widehat{\slnr}$ as central extensions of $\stnr$, while $\stnr$ being central extensions of $\slnr$. When $n\geq 5$, this point of view is clearly demonstrated in (see \cite{BlochLiealgebra} and \cite{Kassellodaycentral}):
\subsection{Steinberg Lie algebras}

In the study of universal central extensions of Chevalley groups, R.Steinberg introduced Steinberg  groups (see \cite[Chapter 6,7]{steinberg1967lectures}) using generators and commutation relations extracted from those of root subgroups of Chevalley groups.  A similar method allows one to define Steinberg Lie algebras (see \cite{BermanMoodyInventiones},\cite{faulkner1989barbilian}):

\begin{definition}\label{defst}
Let $n \geq 3$ be an integer, $R$ an associative $k$-algebra.  The Steinberg Lie algebra $\stnr$ is the Lie algebra generated by $\{\widehat{X}_{ij}(s) \mid s\in R, 1\leq i \neq j \leq n\}$,  subject to the relations 
\begin{align*} &\alpha\mapsto \widehat{X}_{ij}(\alpha) \text{ is a $k$-linear map,}\\
&[\widehat{X}_{ij}(\alpha), \widehat{X}_{jk}(\beta)] = \widehat{X}_{ik}(\alpha\beta), \text{ for distinct } i, j, k, \\
&[\widehat{X}_{ij}(\alpha), \widehat{X}_{kl}(\beta)] = 0, \text{ for } j\neq k, i\neq l,
\end{align*} for all $\alpha, \beta \in R$.
\end{definition} It is known that $\stnr$ extends $\slnr$ centrally, albeit not necessarily universally centrally.

\begin{theorem}\label{structure thm} If $n\geq 5$,  then $\phi: \stnr \to \mathfrak{sl}_n(R)$ gives the universal central extension 
	of $\slnr$. In other words, $  \stnr \cong  \widehat{\slnr}$.  
\end{theorem}

The glaring omission of $n=3$ or $4$ from Theorem \ref{structure thm} is not temporary. Indeed,  universal central extensions of these two exceptional cases are determined in \cite{GaoShang}: they are central extensions of $\mathfrak{st}_n(R)$, where the extension part (which is in general not zero)  is constructed from six copies of $R$. 

\subsection{$\widehat{\slnr}$ for $n=3$ and $n=4$}\label{slnrhat3and4}

Assuming $R$ to be a $k$-algebra that is also free as a $k$-module, the detailed structure of $\widehat{\slfr}$ and $\widehat{\sltr}$ is characterized as follows:

\begin{theorem}[See \cite{GaoShang}]\label{sl4hat} The universal central extension of $\slfr$ is a split central extension of $\stfr$ by $\mathcal{W}$, where $\mathcal{W}$ is the direct sum of six copies of $R_2:=\frac{R}{(2R+R[R,R])}$. The six copies are indexed by orbits of the 24 permutations of $\{1,2,3,4\}$ under the permutation action of the Klein four group $\{(1), (13),  (24), (13)(24)\}$ on the ordering, and we denote each of them by $\epsilon_{ijkl}$, where permutation of the subscripts $i,j,k,l$ under the Klein four group action as described before gives the same copy of $R_2$. In short, $\widehat{\slfr} \cong \stfr \bigoplus  (R_2)^6 $.
	
	$\widehat{\slfr}$ is generated by the symbols $\widehat{X}_{ij}(s)$ where $s \in R$ and $1 \leq i \neq j \leq 4$, and the abelian lie algebra $\mathcal{W}$ subject to the relations
	\begin{align} 
	&[ \mathcal{W}, \mathcal{W}] = [\widehat{X}_{ij}
	(a), \mathcal{W}] = 0,\\
	&a\mapsto \widehat{X}_{ij}(a) \text{ is a $K$-linear map,}\\
	&[\widehat{X}_{ij}(a), \widehat{X}_{jk}(b)] = \widehat{X}_{ik}(ab), \text{ for distinct } i, j, k, \\
	&[\widehat{X}_{ij}(a), \widehat{X}_{kl}(b)] = 0, \text{ for } j\neq k, i\neq l,\\
	&[\widehat{X}_{ij}(a), \widehat{X}_{kl}(b)] = \epsilon_{ijkl}(\overline{ab}), \text{ for } j, k, i, l,\text{ mutally distinct}
	\end{align}
	where $ a, b\in R,$ $ 1\leq i, j, k, l \leq 4$, with $\overline{a}$ denoting the image of $a$ under the projection map $R \rightarrow R_2$.
	
\end{theorem}

\begin{theorem}[See \cite{GaoShang}]\label{sl3hat}

The universal central extension of $\sltr$ is a split central extension of $\sttr$ by $\mathcal{W}$: here $\mathcal{W}$ is the direct sum of six copies of $R_3:=\frac{R}{(3R+R[R,R])}$, where the six copies are indexed by subscripts $\{ijpq \mid 1\leq i,j,p,q\leq 3 , pq=ik \mbox{ or } kj \}$, and we denote each of them by $\epsilon_{ijkl}$ . In short, $\widehat{\sltr} \cong \sttr \bigoplus  (R_3)^6 $, 

This gives rise to the following description: $\widehat{\sltr}$ is generated by the symbols $\widehat{X}_{ij}(s)$ where $s \in R$ and $1 \leq i \neq j \leq 3$, and the abelian lie algebra $\mathcal{W}$ subject to the relations
\begin{align} 
&[ \mathcal{W}, \mathcal{W}] = [\widehat{X}_{ij}
(a), \mathcal{W}] = 0,\\
&a\mapsto \widehat{X}_{ij}(a) \text{ is a $K$-linear map,}\\
&[\widehat{X}_{ij}(a), \widehat{X}_{jk}(b)] = \widehat{X}_{ik}(ab), \text{ for distinct } i, j, k, \\
&[\widehat{X}_{ij}(a), \widehat{X}_{ij}(b)] = 0, \\
&[\widehat{X}_{ij}(a), \widehat{X}_{kl}(b)] = \epsilon_{ijkl}(\overline{ab}), \text{ for } i=k \mbox{ or } j=l.
\end{align}
where $ a, b\in R,$ $ 1\leq i, j, k, l \leq 3$, with $\overline{a}$ denoting the image of $a$ under the projection map $R \rightarrow R_3$.
\end{theorem}

\subsection{``Diagonal'' elements in $\widehat{\slnr}$}

Adjoint actions of $T_{ij} \in \widehat{\mathfrak{sl}_n(R)}$ will be used repeatedly in this paper. These elements are to be understood as``diagonal''  , as they are specific lifts of the diagonal elements of ${\mathfrak{sl}_n(R)}$.

\begin{notation}\label{def of Tij}
	Denote in $\widehat{\mathfrak{sl}_n(R)}$ the element $[ \widehat{X}_{ij}(a), \widehat{X}_{ji}(b)]$ by $T_{ij}(a,b)$, and  by $t(a, b)$  the element  $T_{1j}(a, b) - T_{1j}(1, ba)$. The definition of $t(a,b)$ is independent of $j$. We define similarly $t_i(a,b)$ as the element $T_{ij}(a, b) - T_{ij}(1, ba)$, any $1\leq i \neq j\leq n$.
\end{notation}

To use these elements effectively, we collect below a pool of formulas: 
\begin{align} & T_{ij}(a, b)= -T_{ji}(b, a) \notag\\
&[T_{ij}(a, b), \widehat{X}_{kl}(c)]=0 \text{ for distinct }i, j, k, l\notag\\
&[T_{ij}(a, b), \widehat{X}_{ik}(c)]= \widehat{X}_{ik}(abc), \ \ [T_{ji}(b, a), \widehat{X}_{ki}(c)]= \widehat{X}_{ki}(cab) \notag \\
&[T_{ij}(a, b), \widehat{X}_{ij}(c)]= \widehat{X}_{ij}(abc + cba) \notag 
\\ &[t(a, b), \widehat{X}_{1i}(c)]= \widehat{X}_{1i}((ab-ba)c), \ \ [t(a, b), \widehat{X}_{i1}(c)] = -\widehat{X}_{i1}(c(ab-ba))\notag \\
&[t(a, b), \widehat{X}_{jk}(c)]=0
\text{ for }j, k\geq 2 \label{t equations}\\
& t(a,b) \mbox{ is central if $R$ is commutative} \notag
\end{align}
These equalities are easily verifiable in $\slnr$. In the case of universal central extensions, they follow from presentations of $\widehat{\mathfrak{sl}_n({R})}$, given as in Definition \ref{defst}, Theorems \ref{sl4hat} and \ref{sl3hat}.

\subsection{Improving the presentation of $\widehat{\slnr}$}

\begin{lemma}\label{4pres}
	For $R$ an associative $k$-algebra that is also a free module over $k$, $ \widehat{\slfr}$, as a $k$-Lie algebra, is generated by $\widehat{X}:=\{\widehat{X}_{ij}(s) \mid s\in R, 1\leq i \neq j \leq 4\}$,  subject to the relations
	\begin{align*} 
	&a\mapsto \widehat{X}_{ij}(a) \text{ is a $k$-linear map,}\\
	&[\widehat{X}_{ij}(a), \widehat{X}_{jk}(b)] = \widehat{X}_{ik}(ab), \text{ for distinct } i, j, k, \\
	&[\widehat{X}_{ij}(a), \widehat{X}_{ij}(b)] = 0, [\widehat{X}_{ij}(a), \widehat{X}_{il}(b)] = 0,  \   [\widehat{X}_{ji}(a), \widehat{X}_{li}(b)] = 0  \ \text{ for } j\neq i\neq l,\\
		&[\boldsymbol{w}, [\widehat{X}_{ij}(a), \widehat{X}_{kl}(b)]] = 0, \text{ for } j, k, i, l\text{  mutally distinct, } \text{ for all }  \boldsymbol{w} \in \widehat{X},
	\end{align*}	where $ a, b\in R,$ $ 1\leq i, j, k, l \leq 4$.
\end{lemma}

\begin{proof}
	Denote by $\mathfrak{g}$ the Lie algebra defined as in the statement. Comparing with Theorem \ref{sl4hat}, it is clear that we need to derive the relations (1)-(4) below from the defining relations. We have (1) $[\widehat{X}_{ij}(a), \widehat{X}_{kl}(b)]=[\widehat{X}_{ij}(b), \widehat{X}_{kl}(a)]$  for distinct $i,j,k,l$: indeed, \begin{align*}
	[\widehat{X}_{ij}(a), \widehat{X}_{kl}(b)]&=[[\widehat{X}_{ik}(1),[\widehat{X}_{kl}(a),\widehat{X}_{lj}(1)]] ,\widehat{X}_{kl}(b)]\\
	&=[[\widehat{X}_{ik}(1),\widehat{X}_{kl}(b)], [\widehat{X}_{kl}(a),\widehat{X}_{lj}(1)]]+[\widehat{X}_{ik}(1),[[\widehat{X}_{kl}(a),\widehat{X}_{lj}(1)],\widehat{X}_{kl}(b)]\\
	&=[[\widehat{X}_{ik}(1),\widehat{X}_{kl}(b)], [\widehat{X}_{kl}(a),\widehat{X}_{lj}(1)]]	\\
	&=[\widehat{X}_{il}(b), [\widehat{X}_{kl}(a),\widehat{X}_{lj}(1)]]=0+[\widehat{X}_{kl}(a),[\widehat{X}_{il}(b), \widehat{X}_{lj}(1)]]\\
	&=[\widehat{X}_{kl}(a),\widehat{X}_{ij}(b)].
	\end{align*}
	
	(2) The expression $[\widehat{X}_{ij}(a), \widehat{X}_{kl}(b)]$, where  $i,j,k,l$ are distinct, is invariant under permutation of $i,j,k,l$ under the Klein four group: indeed , without loss of generality, we  set $i,j,k,l$ to be equal to $1,2,3,4 $, respectively. Thanks to (1), it suffices for us to show that $[\widehat{X}_{12}(a), \widehat{X}_{34}(b)]=[\widehat{X}_{14}(a), \widehat{X}_{32}(b)]$, which is clear as $[\widehat{X}_{12}(a), \widehat{X}_{34}(b)]=[[\widehat{X}_{14}(a), \widehat{X}_{42}(1)] ,\widehat{X}_{34}(b)]=0+[\widehat{X}_{14}(a),[\widehat{X}_{42}(1) ,\widehat{X}_{34}(b)]]=[\widehat{X}_{14}(a), \widehat{X}_{32}(b)]$.

	(3) We have $[\widehat{X}_{ij}(ab), \widehat{X}_{kl}(c)]=[\widehat{X}_{ij}(a), \widehat{X}_{kl}(bc)]$: indeed,
	\begin{align*}
	[\widehat{X}_{ij}(ab), \widehat{X}_{kl}(c)]&=[[\widehat{X}_{ik}(a),\widehat{X}_{kj}(b)], \widehat{X}_{kl}(c)]\\
	&=[[\widehat{X}_{il}(ac), \widehat{X}_{kj}(b)]+0\\
	&=[[\widehat{X}_{ij}(a),\widehat{X}_{jl}(c)],\widehat{X}_{kj}(b)]	\\
	&=0+[\widehat{X}_{ij}(a),[\widehat{X}_{jl}(c), \widehat{X}_{kj}(b)]]=[\widehat{X}_{ij}(a),\widehat{X}_{kl}(bc)].
	\end{align*}
	
	(4) We have $2[\widehat{X}_{ij}(a), \widehat{X}_{kl}(b)]=0$ for $i,j,k,l$ distinct:
	
	This is because  $0=[[\widehat{X}_{ij}(1),\widehat{X}_{ji}(1)] , [\widehat{X}_{ij}(a), \widehat{X}_{kl}(b)]]=[\widehat{X}_{ij}(2a), \widehat{X}_{kl}(b)]+0=2[\widehat{X}_{ij}(a), \widehat{X}_{kl}(b)].$
	
	Defining relations of $\mathfrak{g}$ implies that it surjects onto $\widehat{\mathfrak{sl}_4({R})}$,  through $\widehat{X}_{ij}(s) \mapsto \widehat{X}_{ij}(s)$, and that $\mathfrak{g}$ is a central extension of $\widehat{\mathfrak{sl}_4({R})}$ under the same map. Thus $\mathfrak{g}$ and $\widehat{\mathfrak{sl}_4({R})}$ are forced isomorphic, for universal central extensions are centrally closed (see \cite[Theorem 1.8]{Neher}).\end{proof}

\begin{lemma}\label{3pres}
	For an associative $k$-algebra $R$ that is a free module over $k$, the $k$-Lie algebra,$ \widehat{\sltr}$,  is presented by generators $\widehat{X}=\{\widehat{X}_{ij}(s) \mid s\in R, 1\leq i \neq j \leq 3\}$ and relations
	\begin{align*} 
	&a\mapsto \widehat{X}_{ij}(a) \text{ is a $k$-linear map,}\\
	&[\widehat{X}_{ij}(a), \widehat{X}_{jk}(b)] = \widehat{X}_{ik}(ab), \text{ for distinct } i, j, k, \\
	&[\widehat{X}_{ij}(a), \widehat{X}_{ij}(b)] = 0.\\
		&[\boldsymbol{w}, [\widehat{X}_{ij}(a), \widehat{X}_{kl}(b)]] = 0, \text{ for all }    \boldsymbol{w} \in \widehat{X}, \  i=k \text{ or } j=l.
	\end{align*}
	where $ a, b\in R,$ $ 1\leq i, j, k, l \leq 3$.
\end{lemma}

\begin{proof}
	Similarly to Lemma \ref{4pres}, we need to derive several relations from the relations present. Only the ones involving $[\widehat{X}_{ij}, \widehat{X}_{ik}]$ will be tackled: the remaining ones follow by symmetry.
	
	(1) $[\widehat{X}_{ij}(a), \widehat{X}_{ik}(b)]=[\widehat{X}_{ij}(b), \widehat{X}_{ik}(a)]$  for $i,j,k$ distinct: 
	\begin{align*}
	[\widehat{X}_{ij}(a), \widehat{X}_{ik}(b)]&=[[\widehat{X}_{ik}(1),\widehat{X}_{kj}(a)] ,\widehat{X}_{ik}(b)] =0+ [\widehat{X}_{ij}(ba) ,\widehat{X}_{ik}(1)]\\
	&=[\widehat{X}_{ij}(ba), \widehat{X}_{ik}(1)]\\
	&=
	[[\widehat{X}_{ik}(ba),\widehat{X}_{kj}(1)], \widehat{X}_{ik}(1)] =0+[\widehat{X}_{ij}(1), \widehat{X}_{ik}(ba)]\\ 	&=[\widehat{X}_{ij}(1), [\widehat{X}_{ij}(b),\widehat{X}_{jk}(a)]]=0+[   \widehat{X}_{ij}(b),\widehat{X}_{ik}(a)]\\
	&=[   \widehat{X}_{ij}(b),\widehat{X}_{ik}(a)]
	\end{align*}

	(2) $[\widehat{X}_{ij}(ab), \widehat{X}_{ik}(c)]=[\widehat{X}_{ij}(a), \widehat{X}_{ik}(bc)]$:
	\begin{align*}
	[\widehat{X}_{ij}(ab), \widehat{X}_{ik}(c)]&=[ \widehat{X}_{ij}(ab),[\widehat{X}_{ij}(1),\widehat{X}_{jk}(c)]]\\
	&=0+[\widehat{X}_{ij}(1), \widehat{X}_{ik}(abc)]\\
	&=[ \widehat{X}_{ij}(a),[\widehat{X}_{ij}(1),\widehat{X}_{jk}(bc)]]\\
	&= [\widehat{X}_{ij}(a), \widehat{X}_{ik}(bc) ]
	\end{align*}
	
	(3) $3[\widehat{X}_{ij}(a), \widehat{X}_{ik}(b)]=0$  :
	
	This is because  $0=[T_{ij}(1,1),[\widehat{X}_{ij}(a), \widehat{X}_{il}(b)]]=[\widehat{X}_{ij}(2a), \widehat{X}_{il}(b)]+[\widehat{X}_{ij}(a), \widehat{X}_{il}(b)]=3[\widehat{X}_{ij}(a), \widehat{X}_{il}(b)].$
	
	As in Lemma \ref{4pres}, a similar universality argument wraps up the proof.\end{proof}

\subsection{Generating $\widehat{X}_{ij}(R)$ through commutators}

 Let $R$ be generated by a finite set  $\mathcal{X}=\{x_i\}$ (where $x_0=1$) . It is clear through the presentations of $\widehat{\mathfrak{sl}_n(R)}$'s above that $\widehat{X}_{ij}(s)$ (where $s=\alpha\beta$ is a monomial ) is presentable as a single commutator involving $\alpha$ and $\beta$, each a monomial of lower degree(in terms of $\mathcal{X}$). In other words, $[\widehat{X}_{ij}(\alpha), \widehat{X}_{jk}(\beta)] = \widehat{X}_{ik}(\alpha\beta)$ is our {\textit{standard}} way of generating $\widehat{X}_{ij}(R)$.
 
As it turns out,  difficulty of finite presentation proofs in the $n=3$ and $4$ cases stems from the lack of enough indices(i.e. the size of $n$). So we might as well resort to another way to present $\widehat{X}_{ij}(s)$ through iterative commutators, involving less indices. 

According to the batch of equations \ref{t equations}, we can obtain $\widehat{X}_{ij}(s)$ through\begin{align}\label{bettergen}
{\centering \begin{matrix}
	[T_{ik}(\alpha, 1), \widehat{X}_{ij}(\beta)]= \widehat{X}_{ij}(\alpha\beta) & \mbox{ or } &   [T_{kj}(\beta, 1), \widehat{X}_{ij}(\alpha)]= \widehat{X}_{ij}(\alpha\beta).
	\end{matrix}}
\end{align}

We remind the reader that \textbf{\textit{when $R$ is commutative}}, an even better presentation of $\widehat{X}_{ij}(s)$ appears (when $2$  is invertible in $k$) :\begin{align}
{\centering 
	[T_{ij}(\alpha, 1), \widehat{X}_{ij}(\beta)]= \widehat{X}_{ij}(2\alpha\beta).}
\end{align} However we shall not make use of this formula in this paper.

This discussion may be concluded with the slogan ``less indices enables more universal definitions''.

\section{Toward Finite Presentation of $\widehat{\slnr}$}\label{secfinpres}

Recall that $k$ is an arbitrary unital commutative ring. From now on we adopt for the ``root subspaces'' the notation ${X}_{ij}$ instead of $\widehat{X}_{ij}$ for uniformity and simplicity. 

Let us first restate Theorem \ref{mainintro}.

\begin{theorem}\label{main}
	Let $R$ be a finitely generated unital associative $k$-algebra and let $n\geq 3$. The Lie $k$-algebra $ \widehat{\slnr}$ is finitely presented if and only if the $k$-algebra $R$ being finitely presented.
\end{theorem}

Proof of the theorem relies on the following main lemma. We relegate its (lengthy) proof to the next section.
\begin{lemma}\label{mainlemma}

	Let $n\geq 3$ and let $\mathfrak{R}:=k \la \mathcal{X} \ra$ be the free $k$-algebra on a finite set $\mathcal{X}=\{x_1, \ldots ,x_q \}$. Then the Lie $k$-algebra $\widehat{\slnrbar}$ is finitely presented.

\end{lemma}

\begin{remarkstar}
Recall that the presentation of $\widehat{\slnrbar}$ differs greatly between the cases $n\geq 5$, $n=4$ and $n=3$. This suggests that the proof of Lemma \ref{mainlemma} should be divided into three cases. Also, as $k\la \mathcal{X} \ra $ is  free as a $k$-module, we may freely use the presentation given in $\S$\ref{sectionpres}.
\end{remarkstar}

\begin{proof}[Proof of Theorem \ref{main}]\mbox{ }
	
\noindent	``$\Rightarrow$''
Assume that $R$ is obtained by imposing finitely many relations on  $\mathfrak{R}$ (namely  $R=\mathfrak{R}/I$, where $I=\mbox{Ideal}_{\mathfrak{R}} \la \{t_i | i \in S , S \mbox{ finite} \} \ra$). By functoriality of universal central extensions (see \cite[Section 1]{Neher}), one obtains the exact sequence \[ 0 \rightarrow \ker f \rightarrow \widehat{\slnrbar} \stackrel{f}{\rightarrow} \widehat{\slnr} \rightarrow 0, \] where $f$ is the standard projection map, lifting the standard projection map $\phi:\mathfrak{sl}_n(\mathfrak{R}) \rightarrow \mathfrak{sl}_n(R)$ on the $\mathfrak{sl}_n$ level.  We denote the universal extension maps by $\text{p}_\mathfrak{R}: \widehat{\mathfrak{sl}_n(\mathfrak{R})} \rightarrow \mathfrak{sl}_n(\mathfrak{R})$ and $\text{p}_{R}: \widehat{\mathfrak{sl}_n({R})} \rightarrow \mathfrak{sl}_n({R})$.

Consider the ideal $\mathfrak{i}$ of $\widehat{\slnrbar}$ generated by $\{X_{12}(t_i)\}_{i \in S}$. As $\mathfrak{i}$ is clearly sent to zero under $f$, the map $f$ induces  a surjective map $\bar{f}: \widehat{\slnrbar} / \mathfrak{i} \rightarrow \widehat{\slnr}$. 

\begin{claim}\label{claimone}
	The map $\overline{f}$ as above is a central extension of $\widehat{\slnr}$. 
\end{claim}

\noindent\textit{Proof of Claim \ref{claimone}.} Let $X_{ij}(\mathfrak{r})$, where $\mathfrak{r}  \in  \mathfrak{R}$ (resp. $X_{ij}(r)$, where  ${r}  \in {R}$) be generators of $\widehat{\mathfrak{sl}_n(\mathfrak{R})}$  (resp. $\widehat{\mathfrak{sl}_n(R)}$), chosen such that 

\begin{itemize}
	\item They lift the corresponding elements in $\mathfrak{sl}_n(\mathfrak{R})$ (resp. $\mathfrak{sl}_n(R)$);
	\item $f(X_{ij}(\mathfrak{r}))=X_{ij}({r})$;
	\item $\{X_{ij}(\mathfrak{r})\}$ form a generating set of $\widehat{\mathfrak{sl}_n(\mathfrak{R})}$ satisfying the relations in $\S$\ref{sectionpres}.
\end{itemize}Such a choice is guaranteed by \cite[Section 1]{Neher}.

Now take any element $x\in \ker f$.  

The same relations from $\S\ref{sectionpres}$ allow us to  write $x=\mathcal{\tau}+\sum_{i,j,k}[X_{ij}(a_{ijk}),X_{ji}(b_{ijk})] + X_{ij}(c_{ij}) \in \widehat{\slnrbar}$, where $\tau$ lies in the center of $\widehat{\slnrbar}$, satisfying $\text{p}_\mathfrak{R}(\tau)=0$. As $\oplus_{i,j}[X_{ij}(R),X_{ji}(R)] \bigcap \oplus_{i,j} X_{ij}(R)=0$ in $\slnr$,  the equality $\phi \circ \text{p}_\mathfrak{R}(x)=0$ implies that all $c_{ij}\in I$. This yields a reduction to the case $x-\tau \in \sum_{i<j}[X_{ij} ,X_{ji} ]$. Finally, commuting $x-\tau$ with elements of form $X_{lm}(1)$ allows one to conclude that $\sum_{i,j,k}[X_{ij}(a_{ijk}),X_{ji}(b_{ijk})] $ is central in $\widehat{\slnrbar}/\mathfrak{i}$. This demonstrates centrality of $\ker \bar{f}$ and proves the claim. \qedhere

\vskip .1cm

As $\widehat{\slnr}$ is  centrally closed (see \cite{Neher}), Claim \ref{claimone} implies that $\ker f$, being equal to  $\mathfrak{i}$,  is indeed a finitely generated ideal of the Lie algebra $\widehat{\slnrbar}$. This fact along with Lemma \ref{mainlemma} then yields the forward implication instantly. 

\vskip .2cm

\noindent``$\Leftarrow$'': Recall  customary notations $R=\mathfrak{R}/I$, $\mathfrak{R}=k \la \mathcal{X} \ra$, $I=\mbox{Ideal}_{\mathfrak{R}} \la \{t_i | i \in S\} \ra$, $S$ is an  index set (not necessarily finite). Denote by $L\la \mathfrak{X} \ra $ the free Lie algebra on the alphabet $\mathfrak{X}$. For our purposes, we set $\mathfrak{X}:=\{\mathfrak{X}_{ij,x_k} \mid x_k \in \mathcal{X} \mbox{ pairwise distinct, } x_0=1  \}$ and write down a specific finite presentation of $\widehat{\slnr}
$ in the form of  $L\la \mathfrak{X} \ra /{\ker g}$, where $g$ is defined by  $g(\mathfrak{X}_{ij,x_k})=X_{ij}(x_k+I) \in \widehat{\slnr}$. The definition of $f$ allows us to define $h: L \la \mathfrak{X} \ra \rightarrow \widehat{\slnrbar}$ through $h(\mathfrak{X}_{ij,x_k})=X_{ij}(x_k) \in \widehat{\slnrbar}$, and a factorization\[ \begin{tikzcd}
&L \la \mathfrak{X} \ra  \arrow[r,	two heads, "h"] \ar[rr,out=30,in=-210,"g"]
& \widehat{\slnrbar}  \arrow[r, two heads, "f"] & \widehat{\slnr}\\
&\ker g \arrow[r, two heads]  \arrow[u,hook ]
& \ker f \arrow[u, hook]
\end{tikzcd}
\]where $f$ is the standard projection map from $\widehat{\slnrbar}$ to $\widehat{\slnr}$. This gives $\ker g= g^{-1}(0)=h^{-1} (\ker f)$. So if $\ker f$ is not finitely generated, neither can $\ker g$ be finitely generated as an ideal of $L\la \mathfrak{X} \ra $.

Now assume $I$ is not finitely generated. This implies that $\ker f$ is not finitely generated either: indeed, showing as in our proof of  the forward implication, the ideal  $\ker f$ is generated as an ideal by $\{X_{12}(t_i)\}_{i \in S}$. Recall also that $\ker f \cap X_{12}=X_{12}(I)$. Note that if $\ker f$ itself is finitely generated, we can restrict the subscripts $i$ to be chosen from $S_0$, a finite subset of $S$. It then follows from the relations $[T_{13}(a, b), X_{12}(c)]= X_{12}(abc), \ \ [T_{23}(a, b), X_{12}(c)]= -X_{12}(cab)$ that the finite set $\{t_i\}_{i\in S_0}$ generates $I$, giving a contradiction. 

This finishes our proof for the backward implication, and proves Theorem \ref{main}. \hfill $\qed$ 

\end{proof}

\section{The pivotal Lemma}\label{secpivlemma} We devote this section to the proof of Lemma \ref{mainlemma}, which consists of three parallel statements. Although  uniform treatments will be attempted whenever possible, their proofs will inevitably ramify at technical details.		Also, we adopt for the ``root subspaces'' the notation ${X}_{ij}$ instead of $\widehat{X}_{ij}$ for uniformity and simplicity.

\subsection{Initial analysis} Recall that $\widehat{\slnrbar}$ is generated by the symbols $X_{ij}(s)$, (where $s \in \mathfrak{R}$ and $1 \leq i \neq j \leq n$), and are subject to the relations given in  \S\ref{sectionpres}.

As a preliminary reduction, note that $\widehat{\slnrbar}$ is generated by $x_{ij}(s)$, where $s$ is a word (i.e.  monomial with coefficient 1) in $\mathfrak{R}$; and we can restrict the defining relations from \S\ref{sectionpres} to involve only monomials. Since setting $k_1 X_{ij}(s)+k_2X_{ij}(t)= X_{ij}(k_1s+k_2t)$ invites no ambiguity when $s,t$ are monomials, we may safely ignore relations stating $\alpha\mapsto X_{ij}(\alpha)$ being $k$-linear.

Then we first reduce the generating set to a finite subset. Candidates for the finite generating set are clear: just the elements $X_{ij}(1)$ and $X_{ij}(x_k)$ where  $1\leq k \leq l$. Note that with this set of generators, the elements $X_{ik}(u)$, where $\deg(u) \geq 2$, have to be \underline{defined} recursively through commutators $[[\ldots[X_{ii_1}(a_1),X_{i_1i_2}(a_2)] ,\ldots], X_{i_{m-2}i_{m-1}}(a_{m-1})], X_{i_{m-1}k}(a_{m})]$, where the degree of the $a_i's$ are smaller than $2$; or through recursive adjoint action of $T$'s on $X_{ik}(1)$, as was displayed in  (\ref{bettergen}). A most  crucial point is to show that all such expressions are equal.

\vskip .2cm

\subsection{Setup} \mbox{} The blueprint provided in the initial analysis shall be realized  by
\noindent ``approximating'' $\widehat{\slnrbar}$ through finitely presented Lie algebras $L_m$, each provided by appending relations of bounded degree. 

\begin{remarkstar}
	A similar strategy was used in \cite[\S 3.3]{RehmannSoule} and \cite[\S 3]{KrsticMcCool}, in the study of finite presentability for Steinberg groups.
\end{remarkstar}

We proceed as follows: 

\begin{definition}\label{lmdef}
	Let $m$ be a  positive integer, and $\boldsymbol{L_m}$ the Lie $(k-)$algebra generated by the symbols $X_{ij}(s)$,    \textbf{\textit{where $s$ is a monomial with coefficient 1}} in $\mathfrak{R}$ and $\deg(s) \leq m$, with defining relations being:

\begin{itemize}
	\item[]  When $n \geq 5$: 
\begin{align} 
&[X_{ij}(s_1), X_{jk}(s_2)] = X_{ik}(s_1s_2), \text{ for distinct } i, j, k, \text{where } \deg(s_1s_2) \leq m, \tag{$A^1_m$}\label{A^1m}\\
&[X_{ij}(s_1), X_{kl}(s_2)] = 0, \text{ for } j\neq k, i\neq l, \deg(s_1s_2)\leq m+1.
\tag{$A^2_m$}{}\label{A^2m}
\end{align}

\vskip.2cm
	\item[]When $n=4$: 
\begin{align} 
	&[X_{ij}(s_1), X_{jk}(s_2)] = X_{ik}(s_1s_2), \text{ for distinct } i, j, k, \text{where } \deg(s_1s_2) \leq m, \tag{$B^1_m$}\label{B^1m}\\
	&[X_{ij}(s_1), X_{kl}(s_2)] = 0, \text{ for } i=k \text{ or } j=l , \deg(s_1s_2)\leq m+1,
	\tag{$B^2_m$}\label{B^2m}\\
		&[\boldsymbol{w},[X_{ij}(s_1), X_{kl}(s_2)]] = 0, \text{ for } |\{i,j,k,l\}|=4, \deg(s_1s_2)\leq m+1, \boldsymbol{w}= X_{ij}(\gamma), 
	\tag{$B^3_m$}\label{B^3m}
	\end{align}{ for an arbitrary letter }$\gamma \in \mathcal{X} \cup \{1\}$.
	
	\vskip .2cm
	
	\vskip.2cm
		\item[]When $n=3$: 
\begin{align} 
	&[X_{ij}(s_1), X_{jk}(s_2)] = X_{ik}(s_1s_2), \text{ for distinct } i, j, k, \text{where } \deg(s_1s_2) \leq m, \tag{$C^1_m$}\label{C^1m}\\
	&[X_{ij}(s_1), X_{ij}(s_2)] = 0, \text{ for } i\neq j, \  \deg(s_1s_2)\leq m+1,
	\tag{$C^2_m$}\label{C^2m}\\
	&[\boldsymbol{w},[X_{ij}(s_1), X_{kl}(s_2)]] = 0, \text{ for } i=k \text{ or } j=l, \deg(s_1s_2)\leq m+1, \  \boldsymbol{w}= X_{ij}(\gamma), 
	\tag{$C^3_m$}\label{C^3m}
	\end{align}{ for an arbitrary letter }$\gamma \in \mathcal{X} \cup \{1\}$.
	
	\vskip .2cm
	 
\end{itemize}

\end{definition}
The following observation is clear, as we're working in $\widehat{\slnrbar}$, $n\geq 3$:

\begin{prop}\label{lmgen}
	The Lie algebra $L_m$ is generated by the finite set $\{X_{ij}(x_k),X_{ij}(1)\}_{1\leq i \neq j\leq n, 1\leq k \leq m}$ for $m \geq 2$.
\end{prop}

\begin{proof}  
	By definition, the Lie algebra $L_m$ is generated by elements of the form $X_{ij}(s)$, where $s$ is a monomial with coefficient 1 in $\mathfrak{R}$ and $\deg(s) \leq m$. Now if $s=t_1 \ldots t_m$ where $deg(t_i) \leq 1$, then $X_{ij}(s)$ can be expressed as an nested commutator. For example, $X_{12}(t_1t_2t_3)=\bigl[[X_{13}(t_1),X_{32}(t_2)] ,X_{23}(t_3)\bigr]$.
\end{proof}

We now describe the approximation process. Define homomophisms $\phi_{m,{m+1}}: L_m \rightarrow L_{m+1}$ by assigning $\phi_{m,m+1}(X_{ij}(s))=X_{ij}(s)$. As the defining relations in $L_m$ hold in $L_{m+1}$, we see that these homomorphisms are always well-defined. By composing such $\phi$, we obtain homomorphisms $\phi_{m,{m+j}}: L_m \rightarrow L_{m+j}$. By Proposition \ref{lmgen} above, we see that all $\phi_{m,{m+j}}$ are surjective. It follows that $\la L_i, \phi_{ij}\ra$ form a direct system over the index set $\dbZ$. The direct limit $\varinjlim L_i$ can be explicity described as the Lie algebra generated by the symbols $X_{ij}(s)$, where $s$ is a monomial with coefficient 1 in $\mathfrak{R}$; 
with defining relations being the union of the defining relations for $L_{m}$.
 
According to our initial analysis, this is exactly $\widehat{\slnrbar}$. Consequently general properties of direct limits imply that if one can show that the maps $\phi_{m,{m+1}}$ become isomorphisms for all $m$ greater or equal to some $T$, then $L_T \cong \widehat{\slnrbar}$.

\section{Unraveling Commutators: length of words and permutations}\label{sec:length-of-words-permutations-and-vanishing-of-commutators}

It is clear that Lemma \ref{mainlemma} will be proven once we show:
\begin{prop}\label{claim2}
	 The homomorphism $\phi_{m,{m+1}}$ is an isomorphism for large sufficiently $m$. More specifically:
	 
	 \begin{itemize}
	 	\item when
	 	 $n\geq 5$,  $m \geq 4$;  
	 	 \item  {when  $n=4$, $m\geq 10$};
	 	 	
	 	 	\item when   $n=3$, $m\geq  \max\{q+3,10\}$, where $q=|\mathcal{X}|$, the rank of $\mathfrak{R}$.
	 \end{itemize}
\end{prop}

\noindent\textit{Proof of Proposition \ref{claim2}.}  We prove this Claim according to the two-step scheme:\begin{itemize}
	\item[($\mathrm{I}$).]\label{I} Define in $L_m$ the elements $X_{ij}(u)$, where $u\in \mathfrak{R}$ is a monomial with coefficient 1 and $deg(u)=m+1$;
	
	\item[($\mathrm{I}\mathrm{I}$).]\label{II} Show that with such $X_{ij}(u)$ defined in $L_{m}$, all the defining relations in $L_{m+1}$ hold in $L_{m}$ as well.
\end{itemize}

\subsubsection*{\textbf{Step ($\mathrm{ I }$)}} Write $u=x v$, where $deg(v)=m$, and $x$ a element of the generating set of $\mathfrak{R}$, defined as in the statement of Lemma \ref{mainlemma}. Borrrowing notations from \S\ref{sectionpres}, we define in $L_m$
\begin{align}\label{def+deg}
 X_{ij}(u)=\ad(T_{ik}(1,x))X_{ij}(v), \ \  k\neq i, k\neq j.\tag{$*$}
\end{align}

This expression is independent of $k$, and  coincides with other reasonable definitions (using either nested commutators or $T$'s) of $X_ij(u)$. To see this, we require a slew of equalities, beginning from

\begin{lemma} (\cite[Lemme 1.13]{Kassellodaycentral}) \label{t eqn} Let $a,b,c$ be monomials in $\mathfrak{R}$, $deg(abc)\leq m$. Then the following equalities, as well as the versions obtained by permutation of indices $1,2, \ldots n$ , hold in $L_m$ for each $m\geq 2$:
	\begin{enumerate}
		\item $T_{12}(ab,c)=T_{13}(a,bc)+T_{32}(b,ca)$
		\item $ T_{12}(1,a)=-T_{21}(1,a)=T_{12}(a,1)$ 
		\item $T_{21}(a,b)-T_{21}(1,ba)=T_{31}(a,b)-T_{31}(1,ba)$
	\end{enumerate}
\end{lemma}

\begin{proof}
	Note that these are standard equalities in $\widehat{\slnrbar}$,  whose proofs involve only the defining relations of $L_m$. We have \begin{align*}
	T_{12}(ab,c)=[X_{12}(ab),X_{21}(c)]&=[[X_{13}(a),X_{32}(b)],X_{21}(c)]\\ &=[[X_{13}(a),X_{21}(c)],X_{32}(b)]+[[X_{13}(a),[X_{32}(b),X_{21}(c)]]\\
	&=[-X_{23}(ca),X_{32}(b)]+[X_{13}(a),X_{31}(bc)]\\
	&=T_{13}(a,bc)+T_{32}(b,ca).
\end{align*} Similarly, (2) follows from (1), and (3) follows from (1) and (2).	
\end{proof}

\begin{remarkstar}
	The proof of Lemma \ref{t eqn}  is a demonstration of the general fact that relations of total degree $\leq m$ in $\widehat{\slnrbar}$ always hold in $L_m$. In fact, all equalities from the pool (\ref{t equations}) obey this rule. For example, if $i,j \geq 2$, then $t(a,b)$ commutes with all elements of the form $X_{ij}(s)$ ``up to degree $m-\deg(ab)$'' (in $L_m$).
\end{remarkstar}

\begin{lemma} The expression \label{def+deg welldef}
	(\ref{def+deg}) is well defined in $L_m(m \geq 3)$. Namely for $u=xv$, $deg(v)=m$, and $deg(x)=1$, the element $X_{ij}(u):=\ad(T_{ik}(1,x))X_{ij}(v), \  k\neq i, k\neq j$  does not depend on $k$.
\end{lemma}

\begin{proof}
	Without loss of generality, we may assume $i=1 , k=2, j=4$. So all there is to prove is $[T_{12}(1,x), x_{14}(v)]=[T_{13}(1,x), x_{14}(v)]$ (when $n=3$, there is nothing to prove). The proof is separated into two cases:
	\begin{enumerate}
		\item[(i)] When $ n\geq 5$: As $m\geq 4$, write $v=yz$, $deg(y)=1$. Now   $[T_{12}(1,x), X_{14}(v)]-[T_{13}(1,x), X_{14}(v)]=[T_{32}(1,x),X_{14}(v)]=[T_{32}(1,x),[X_{15}(y),X_{54}(z)]]$. By relation (\ref{A^2m}), the element $T_{32}(1,x)$ commutes with both entries of the inner bracket of the last term, so we have zero.
		\item[(ii)] When $n=4$: Adopt the  setup as in part (i), while expressing further $z=wv$, where $deg(w)=1$. Then we have  \begin{align*}
		[T_{32}(1,x),X_{14}(v)]&=[T_{32}(1,x),[T_{13}(1,y), X_{14}(z)]]\\&=[[T_{32}(1,x),T_{13}(1,y)], X_{14}(z)]-[T_{13}(1,y),[T_{32}(1,x),X_{14}(z)]]\\&\stackrel{(\mbox{\tiny\ref{B^3m}})}{=}[[T_{32}(1,x),T_{13}(1,y)], X_{14}(z)].
		\end{align*} Noting that $[T_{32}(1,x),T_{13}(1,y)]=[T_{32}(1,x),[X_{13}(1),X_{31}(y) ]]=-T_{13}(x,y)+T_{13}(1,xy)$ commutes with both entries of  $[X_{12}(w),X_{24}(v)]$ (By  (\ref{t equations})), we have zero again.
	\end{enumerate}And this intermediate lemma is proven.
\end{proof}

This defines the element $X_{ij}(u)$ $(deg(u)=m+1)$ in $L_m$. Eventually we have to show that this definition is compatible with all other reasonable ones, as can be seen in $\mathfrak{sl}_n(\mathfrak{R})$. We start from the following proposition, which studies what happens when we multiply on different sides.

\begin{prop}\label{left=right ad}
	  Let $n\geq3$. In $L_m$ with $m\geq 4$, the equality \[ad(T_{ik}(1,x))X_{ij}(zy)=ad(T_{kj}(1,y))X_{ij}(xz)\]holds. Here $u=xzy$ is a word of degree $m+1$ in $\mathfrak{R}$ with {{$ \deg(x)\deg(y) \geq 2$}}, $\deg(z)\geq 2$.
\end{prop}

\begin{proof}\mbox{}A proof  for $n\geq 4$ is rather straightforward: \begin{align*}
	& \ [T_{ik}(1,x),X_{ij}(zy)]\stackrel{\mbox{\tiny(lower degree relations)}}{=} [T_{ik}(1,x),[T_{kj}(1,y),X_{ij}(z)]]\\ =& \  [[T_{ik}(1,x),T_{kj}(1,y)],X_{ij}(z)]+ [T_{kj}(1,y),X_{ij}(xz)].
	\end{align*} So all we have to show is that  $[[T_{ik}(1,x),T_{kj}(1,y)],X_{ij}(z)]=0$. As we work with $n\geq 4$, the last two lines  of the argument given in case (ii) of Lemma \ref{def+deg welldef} apply, with the slight modification that we write $X_{ij}(z)$ as a commutator of $X_{il}$ and $X_{lj}$, where $|\{i,j,k,l\}|=4$.

A proof for $n=3$ is now in order. As it involves only three indices, the proof is general enough to cover all  $n\geq 3$ cases. This promised proof depends on a trick worth singling out:

\subsection{\textbf{The permutation trick}}\label{permtrick}

\begin{lemma}\label{main trick}
	Let the setup   be as in Proposition \ref{left=right ad}, let $\alpha, \beta, \gamma$ be words in $\mathfrak{R}$,  $\deg(\gamma) \geq 3$, $\deg(\alpha\beta\gamma)=m+1$. Then \[[[T_{ik}(1,\alpha),T_{kj}(1,\beta)], X_{ij}(\gamma)]=[[T_{ik}(1,\alpha),T_{kj}(1,\beta)], X_{ij}(\gamma')], \] where $\gamma'$ is any word obtained from permuting the letters in  $\gamma$.
\end{lemma}

\begin{proof}[Proof of Lemma \ref{main trick}] \mbox{}
	
 Recall that $[T_{ij}(1,\alpha),T_{jk}(1,\beta)]=-T_{jk}(\alpha,\beta)+T_{jk}(1,\beta\alpha)$. 	Without loss of generality, we fix $i=1,j=3,k=2$.

 Express $ \gamma=
\gamma_1\gamma_2\gamma_3$, where 
$\deg(\gamma_1)\deg(\gamma_2)\deg(\gamma_3) \geq 1$. Then we have the expression \begin{align*}
 E:=&[[T_{12}(1,\alpha),T_{23}(1,\beta)], X_{13}(z)]\\
 =&[[T_{12}(1,\alpha),T_{23}(1,\beta)], [X_{12}(\gamma_1\gamma_2),X_{23}(\gamma_3)]]\\
 =&[X_{12}(\gamma_1\gamma_2(\alpha \beta-\beta\alpha)),X_{23}(\gamma_3)]+[X_{12}(\gamma_1\gamma_2),X_{23}((\beta\alpha-\alpha \beta)\gamma_3)].
\end{align*}

Recall the notation $t(a, b)=T_{1j}(a, b) - T_{1j}(1, ba)$. Do note:
\begin{itemize}
	\item[($i$)]We may permute $\gamma_1$ and $\gamma_2$ in $E$, since \begin{align*}
 & \quad [X_{12}((\gamma_1\gamma_2-\gamma_2\gamma_1)(\alpha \beta-\beta\alpha)),X_{23}(\gamma_3)]+[X_{12}(\gamma_1\gamma_2-\gamma_2\gamma_1),X_{23}((\beta\alpha-\alpha \beta)\gamma_3)]\\=& \quad [t(\gamma_1,\gamma_2),[X_{12}(\alpha \beta-\beta\alpha),X_{23}(\gamma_3)]]+[t(\gamma_1,\gamma_2),[X_{12}(1),X_{23}((\beta\alpha-\alpha \beta)\gamma_3)]]\\ \stackrel{\mbox{\tiny(\ref{C^1m})}}{=}& \quad 0;
\end{align*}

\item[($ii$)]  the computation in ($i$) may be applied (through symmetry)  to $E=$$\big[[T_{12}(1,\alpha),T_{23}(1,\beta)],$ $[X_{12}(\gamma_1),X_{23}(\gamma_2\gamma_3)]\big]$, allowing a swap of $\gamma_2$ and $\gamma_3$;

\item[($iii$)]   $S_n$ is generated by the elements $(12) $, $(123 \ldots ,n-1) $, $(n-1, n)$,  

\end{itemize}
\vskip.2cm 

It follows  that expressions $\bigl[ [T_{12}(1,\alpha),T_{23}(1,\beta)], X_{13}(z')\bigr]$ are all equal to $E$, where $z'$ is any monomial obtained by 
shuffling the factors of $z$. This proves Lemma \ref{main trick} 
\end{proof}

\noindent\textit{Proof of Proposition \ref{left=right ad}} \textit{continued:} Let $n=3$. Recall that we are trying to prove in $L_m$ (under restriction on degrees of $x,y,z$)  that the equality \[ad(T_{ik}(1,x))X_{ij}(zy)=ad(T_{kj}(1,y))X_{ij}(xz)\] holds.

Without loss of generality,
	   fix $i=1$, $j=3$, $k=2$.  As it is clear that\begin{align*}
	    [T_{23}(1,y),X_{13}(xz)]=[T_{23}(1,y), [T_{12}(1,x),X_{13}(z)]], &\mbox{ and }\\ [T_{12}(1,x),X_{13}(zy)]=[T_{12}(1,x),[T_{23}(1,y), X_{13}(z)]], &
	    \end{align*} our goal becomes proving $[[T_{12}(1,x),T_{23}(1,y)], X_{13}(z)]=0$.  
	    
	    Writing $z=t_1t_2$ (where $\deg(t_1)\deg(t_2) \geq 1$) again,  we have the equality   \begin{align}
	   &\bigl[[T_{12}(1,x),T_{23}(1,y)], X_{13}(t_1t_2)\bigr] \notag\\
	   =&\bigl[ [T_{12}(1,x),T_{23}(1,y)], [X_{12}(t_1),X_{23}(t_2)]\bigr] \tag{$\ast$}\\
	   =&[X_{12}(t_1[x,y]),X_{23}(t_2)]-[X_{12}(t_1),X_{23}([x,y]t_2)].\notag
	   \end{align} As the defining relations of $L_m$ imply the degree $m+1$ expressions $[X_{13}(t_1[x,y]),X_{23}(t_2)]$ and $[X_{12}(t_1),X_{13}([x,y]t_2)]$ being either zero or central , we may continue equality $(\ast)$ by \begin{align*}
	   &[X_{12}(t_1[x,y]),X_{23}(t_2)]-[X_{12}(t_1),X_{23}([x,y]t_2)]\\=&[X_{13}(t_1[x,y]),T_{32}(1,t_2)] -[T_{12}(1,t_1),X_{13}([x,y]t_2)]\\
	   =&\bigl[[T_{12}(1,t_1),X_{13}([x,y])],T_{32}(1,t_2)\bigr]-[T_{12}(1,t_1),X_{13}([x,y]t_2)]\\
	   =&\bigl[[T_{12}(1,t_1),T_{32}(1,t_2)], X_{13}([x,y])\bigr]\\
	   =&-\bigl[[T_{12}(1,t_1),T_{23}(1,t_2)], X_{13}([x,y])\bigr]. 
	   \end{align*} Since $\deg(xy)\geq  3$, applying Lemma \ref{main trick} gives $\bigl[[T_{12}(1,t_1),T_{23}(1,t_2)], X_{13}([x,y])\bigr]=0$, and concludes the proof of Proposition \ref{left=right ad}. \end{proof}	
   
The following corollary relaxes the condition $\deg(x)\deg(y)\geq2$ in Proposition \ref{left=right ad} .
   
\begin{corollary}\label{any  ad works}
	
		  Let $n=3$. In $L_m$ with $m\geq 4$, the equality \[ad(T_{ik}(1,x))X_{ij}(zy)=ad(T_{kj}(1,y))X_{ij}(xz)\]holds. Here $u=xzy$ is a word of degree $m+1$ in $\mathfrak{R}$ with {{$ \deg(x)\deg(y) \geq 1$}}, $\deg(z)\geq 2$.
\end{corollary}

\begin{proof}
Write $z=z_1z'z_2$, where $deg(z_i)=1$. Then this corollary is a consequence of the observation $x(z_1z'z_2y)=(xz_1z')(z_2y)=(xz_1)(z'z_2y)=(xz_1z'z_2)y$ and repeated applications of Proposition \ref{left=right ad}.
\end{proof}

\begin{remark}
	From the proof of Corollary \ref{any  ad works}, we may extract the stronger equality \newline $ad(T_{ik}(1,x))X_{ij}(zy)=$$ad(T_{kj}(1,z_2y))X_{ij}(xz_1z')=$$ad(T_{ik}(1,xz_1))X_{ij}(z'z_2y)=$$ad(T_{kj}(1,y))X_{ij}(xz)$ in $L_m$. 
\end{remark}

   This concludes our proof of step ($\mathrm{I}$), and we are now left with step $(\mathrm{I}\mathrm{I})$, the last and most involved part of the proof. 
   
\begin{remark}\label{pfmoral}
	The rule of thumb regarding Proposition \ref{claim2} is that the proof gets harder when $n$ gets smaller: this applies to both the previous, and the upcoming part of the proof.  When $n \leq 4$, we'll have to impose lower bounds on the degree of commutators to make the arguments work. We also remark that analysis of the relation between commutators like $[X, X_{kl}(b)]$ and $[X, X_{kl}(b')]$ turns out to be most important, where $b'$ is obtained through permuting the factors of $b$. This is exactly where the permutation trick, discussed in this section (\ref{permtrick}), becomes useful again.
\end{remark}

\section{Completing the proof}\label{partb}

We recall what has been done so far: by Section \ref{sectionpres}, we are equipped with presentations of $\widehat{\mathfrak{sl}_n(\mathfrak{R})}$, divided into three cases $n\geq 5$, $n=4$ and $n=3$; in Section \ref{secfinpres}, we reduced the finite presentability problem of general $\widehat{\mathfrak{sl}_n(R)}$ to that of $\widehat{\mathfrak{sl}_n(\mathfrak{R})}$;  in Section \ref{secpivlemma}, we adopted the strategy of approximating $\widehat{\mathfrak{sl}_n(\mathfrak{R})}$ by finitely presented Lie algebras $L_m$, where $\varinjlim L_m \cong \widehat{\mathfrak{sl}_n(\mathfrak{R})}$ and reduced the problem to showing $L_m$ stabilizes for large enough $m$ (i.e. Proposition \ref{claim2}); in Section \ref{sec:length-of-words-permutations-and-vanishing-of-commutators} we showed how to define the degree $m+1$ elements in $L_m$.

In this final section, we proceed to finish the proof of Proposition \ref{claim2}.

\subsubsection*{\textbf{Step ($\mathrm{I}\mathrm{I}$) of Proposition \ref{claim2}}}

Recall the relations (\ref{A^1m})  through (\ref{C^3m}) introduced in Definition \ref{lmdef}. We specifically want to show that (depending on the structure of $\widehat{\slnrbar}$)  once the degree $m+1$ elements are defined (in a specific approximating Lie algebra $L_{\bullet}$), then the collection ($A_{m+1}$) (or its $B$,$C$ counterpart, respectively) follows from the group of relations of smaller degree, namely ($A_m$) (or its $B$,$C$ counterpart, respectively). In this way one establishs $L_m \cong L_{m+1}$ (through the standard identification of degree zero and one generators); it is clear the same argument derives that all $L_i$, $i\geq m$ are isomorphic. 

Let us prove these relations one by one, for cases $n\geq 5$, $n=4$ and $n=3$, respectively. Do note that we only have to establish the relations in ($A_{m+1}$) (or its $B$,$C$ counterpart, respectively) that involve elements of degree $m+1$, or when the degrees of elements on the bracket side of the relation sum to $m+2$.  

\noindent\textit{\textbf{Convention} }We agree that, unless otherwise stated, when referring to elements of $\mathfrak{R}$,  a single letter (with or without subscripts) always represent a word. 

\begin{lemma}\label{relm+1}
	The following relations hold in $L_m$.
	
	\begin{itemize}
		\item[]  When $n \geq 5$: 
		\begin{align} 
		&[X_{ij}(s_1), X_{jk}(s_2)] = X_{ik}(s_1s_2), \text{ for distinct } i, j, k, \text{where } \deg(s_1s_2) \leq m+1, \tag{$A^1_{m+1}$}\label{A^1_{m+1}}\\
		&[X_{ij}(s_1), X_{kl}(s_2)] = 0, \text{ for } j\neq k, i\neq l, \deg(s_1s_2)\leq m+2.
		\tag{$A^2_{m+1}$}{}\label{A^2_{m+1}}
		\end{align}

		\vskip.2cm
		\item[]When $n=4$, $m\geq 10$: 
		\begin{align} 
		&[X_{ij}(s_1), X_{jk}(s_2)] = X_{ik}(s_1s_2), \text{ for distinct } i, j, k, \text{where } \deg(s_1s_2) \leq m+1, \tag{$B^1_{m+1}$}\label{B^1_{m+1}}\\
		&[X_{ij}(s_1), X_{kl}(s_2)] = 0, \text{ for } i=k \text{ or } j=l , \deg(s_1s_2)\leq m+2,
		\tag{$B^2_{m+1}$}\label{B^2_{m+1}}\\
		&[\boldsymbol{w},[X_{ij}(s_1), X_{kl}(s_2)]] = 0, \text{ for } |\{i,j,k,l\}|=4, \deg(s_1s_2)\leq m+2, \boldsymbol{w}= X_{ij}(\gamma),
		\tag{$B^3_{m+1}$}\label{B^3_{m+1}}
		\end{align} { for an arbitrary letter }$\gamma \in \mathcal{X} \cup \{1\}$.
		
		\vskip.2cm
		\item[]When $n=3$, $m\geq  \max\{|\mathcal{X}|+3, 10\}$ : 
		\begin{align} 
		&[X_{ij}(s_1), X_{jk}(s_2)] = X_{ik}(s_1s_2), \text{ for distinct } i, j, k, \text{where } \deg(s_1s_2) \leq m+1, \tag{$C^1_{m+1}$}\label{C^1_{m+1}}\\
		&[X_{ij}(s_1), X_{ij}(s_2)] = 0, \text{ for } i\neq j, \  \deg(s_1s_2)\leq m+2,
		\tag{$C^2_{m+1}$}\label{C^2_{m+1}}\\
		&[\boldsymbol{w},[X_{ij}(s_1), X_{kl}(s_2)]] = 0, \text{ for } i=k \text{ or } j=l, \deg(s_1s_2)\leq m+2, \  \boldsymbol{w}= X_{ij}(\gamma), \tag{$C^3_{m+1}$}\label{C^3_{m+1}}
		\end{align}{ for an arbitrary letter }$\gamma \in \mathcal{X} \cup \{1\}$.
		
	\end{itemize}
\end{lemma}

\subsection*{\textit{\textbf{Proof of Lemma \ref{relm+1}}},   $\boldsymbol{ n\geq 5}$:} The two types of relations we need to show come from $\mathfrak{st}_n$. We prove them one by one, while assuming (\ref{A^1m}) and (\ref{A^2m}), and having degree $m+1$ elements  defined by (\ref{def+deg}) (see Step ($\mathrm{I}$) of Proposition \ref{claim2}).
	
	\vskip .3cm
	\begin{itemize}
		\item[$\bullet$$(A^1_{m+1})$:]
		\[[X_{ij}(s), X_{jk}(t)] = X_{ik}(st), \text{ for distinct } i, j, k, \text{where } \deg(st)  \leq m+1.\]
		
		\noindent\textit{Indeed,} without loss of generality, assume $\deg(s) \geq \deg(t)$. Corollary \ref{any  ad works} allows for considering only the case $s=xy$, where $deg(x)=1$. Take $l$ different from $i,j,k$. Then 
		\begin{align*}&[X_{ij}(xy), X_{jk}(t)]\\=&[[T_{il}(1,x),X_{ij}(y)],X_{jk}(t)]\\\stackrel{\mbox{\tiny(\ref{A^2m})}}{=}& 0 + [T_{il}(1,x),[X_{ij}(y),X_{jk}(t)]]\\ \stackrel{\mbox{\tiny(\ref{A^1m})}}{=}& \ad (T_{il}(1,x)) X_{ik}(yt)=X_{ik}(st).
		\end{align*}This proves $(A^1_{m+1})$.
		
		\vskip.2cm
		\item[$\bullet$($A^2_{m+1}$):] \[[X_{ij}(s), X_{kl}(t)] = 0, \text{ for } j\neq k, i\neq l, \deg(s_1s_2)\leq m+2.\]
		
	\noindent\textit{Indeed,}  it suffices to consider the case $i=k$ and $j\neq l$: assume $s=yx$, where $deg(x)=1$,  and $m$ different from $i,j,l$. Then
\begin{align*}
 &[X_{ij}(yx), X_{il}(t)]\\\stackrel{\mbox{\tiny Corollary \ref{any  ad works}}}{=}&[[T_{mj}(1,x),X_{ij}(y)],X_{il}(t)]\\\stackrel{\mbox{\tiny (\ref{A^2m})}}{=}&[T_{mj}(1,x),[X_{ij}(y),X_{il}(t)]]-0= 0. 
\end{align*}
Two other cases ($j=l,i\neq k$ and $j=l,i= k$) clearly follow. The case $[X_{ij}(s), X_{kl}(t)]$ ($|\{i,j,k,l\}|=4$) follows from using $X_{kl}=[X_{km},X_{ml}]$, where $|\{i,j,k,l,m\}|=5$.
 
	\end{itemize}
	
\subsection*{\textit{\textbf{Proof of Lemma \ref{relm+1}}},   $  \boldsymbol{ n=4}$:}

Recall that $m\geq 10$, and $\mathfrak{R}:=k \la \mathcal{X} \ra$.  We have three types of relations to show, while assuming (\ref{B^1m}) through (\ref{B^3m}) and having degree $m+1$ elements  defined as before. 

\vskip.3cm
\begin{itemize}
	\item[$\bullet$$(B^1_{m+1})$:]
\[	[X_{ij}(s), X_{jk}(t)] = X_{ik}(st), \text{ for distinct } i, j, k, \text{where } \deg(st) \leq m+1.\]

	\noindent	\textit{Indeed,} the proof is similar to that of $(A^1_{m+1})$.
		\vskip .2cm 
	\item[$\bullet$$(B^2_{m+1})$:]
	\[[X_{ij}(s), X_{kl}(t)] = 0, \text{ for } i=k \text{ or } j=l , \deg(st)\leq m+2.\]
	
		\noindent\textit{Indeed,}  after replacing appearances of (\ref{A^2m}) by (\ref{B^3m}), we get the same expression as the first three cases of $(A^2_{m+1})$.
		\vskip .2cm 
	\item[$\bullet$$(B^3_{m+1})$:]
	\[[\boldsymbol{w},[X_{ij}(s), X_{kl}(t)]] = 0, \text{ for } |\{i,j,k,l\}|=4, \deg(st)\leq m+2, \boldsymbol{w}= X_{ij}(\gamma), \]{ for an arbitrary letter }$\gamma \in \mathcal{X} \cup \{1\}$.
	
	\vskip .2cm
		\noindent\textit{Indeed,} take any $\boldsymbol{w} \in \{X_{l_1l_2}(\tau) \mid \deg{\tau}\leq 1 ,  l_1 \neq l_2 \in \{1,2,3,4\}\}$. 
	
	If $s=xy, \ \deg(y)=1, \ deg(s)=m+1$,  the computation \begin{gather*}
	[X_{ij}(xy), X_{kl}(t)]\stackrel{\mbox{\tiny Proposition  (\ref{left=right ad})}}{=}\big[[T_{kj}(1,y),X_{ij}(x)], X_{kl}(t)\big]\\ \stackrel{\mbox{\tiny (\ref{B^3m})}}{=}\big[X_{ij}(x), [T_{jk}(1,y),X_{kl}(t)] \big]
	\end{gather*} shows that we can reduce the assertion to the case where $\deg(s) < m+1$;  moreover we may assume additionally  {$\deg(s)\geq6,\  \deg(t)\geq6$}. 
	
	The remaining treatment depends on three cases regarding $\boldsymbol{w}$:
	
	\begin{enumerate}
		\item When $\boldsymbol{w}  \in {X}_{12} \cup {X}_{34} \cup {X}_{14} \cup {X}_{32}$, the relation follows from 
		$(B^3_{m})$ and $(B^2_{m})$ by a direct application of the Jacobi identity;
		
		\item  When $\boldsymbol{w}  \in {X}_{23} \cup {X}_{41}$, it suffices to consider the former case. Let  $s=t_1u_1, t=t_2u_2$, $\deg(t_1)=\deg(t_2)={ 3 },  { 3 } \leq \deg(u_1), \deg(u_2) \leq m-3, \deg(t_1u_1t_2u_2) \leq m+2$.
		{ {In addition,  let $t_2=s_a h$, where $\deg (s_a)=1, \deg (h)=2 $}}. Then 
		\begin{align*}
		&[{X}_{23}(\gamma), [{X}_{12}(t_1u_1), {X}_{34}(t_2u_2)]]\\
		=& [{X}_{23}(\gamma), [[{X}_{13}(t_1), {X}_{32}(u_1)], [{X}_{32}(t_2),{X}_{24}(u_2)]]\\
		\stackrel{\mbox{\tiny(\ref{B^2m})}}{=}& [{X}_{23}(\gamma), [{X}_{32}(t_2),[[{X}_{13}(t_1), {X}_{32}(u_1)], {X}_{24}(u_2)] ]]   \\
		=& [{X}_{23}(\gamma), [{X}_{32}(t_2),{X}_{14}(t_1u_1u_2)]]\\
		\stackrel{\mbox{\tiny (\ref{B^3m})}}{=}& [[{X}_{23}(\gamma), {X}_{32}(t_2)] ,{X}_{14}(t_1u_1u_2)]]   \\
		=& [T_{23}(\gamma, t_2) ,[{X}_{12}(t_1), [{X}_{21}(u_1),{X}_{14}(u_2)]]]\\
		\stackrel{\mbox{\tiny(\ref{B^3m})}}{=}& [-{X}_{12}(t_1\gamma t_2), [{X}_{21}(u_1),{X}_{14}(u_2)]]\\
		+&[{X}_{12}(t_1), [{X}_{21}(\gamma t_2u_1),{X}_{14}(u_2)]]+0\\
		=&{  [T_{23}(\gamma s_a , h)  ,[{X}_{12}(t_1), [{X}_{21}(u_1),{X}_{14}(u_2)]]]}\\
		=& [[{X}_{23}(\gamma s_a), {X}_{32}(h)] ,{X}_{14}(t_1u_1u_2)]]\\
		=& [{X}_{23}(\gamma s_a), [{X}_{32}(h),{X}_{14}(t_1u_1u_2)]]
		\stackrel{\mbox{\tiny(\ref{B^3m})}}{=}0. 
		\end{align*}
		\item If $\boldsymbol{w} \in {X}_{31} \cup {X}_{24} \cup {X}_{13} \cup  {X}_{42} \cup {X}_{21} \cup {X}_{43}$, then $\boldsymbol{w}$ can be represented as commutators in $X_{ij}$'s treated in cases $(1),(2)$. This completes the proof of the assertion. \end{enumerate}
	\vskip .2cm
	As we have traversed all cases, this proves $(B^3_{m+1})$.
\end{itemize}

\subsection*{\textit{\textbf{Proof of Lemma \ref{relm+1}}},   $  \boldsymbol{ n=3}$:} Recall that $\mathfrak{R}:=k\la \mathcal{X} \ra $. Under the premise $m\geq  \max\{|\mathcal{X}|+3, 10\}$, we again have three types of relations to show, while assuming (\ref{C^1m}) through (\ref{C^3m}) and having degree $m+1$ elements  defined as before. 

\vskip.3cm
\begin{itemize}
	
	\vskip .2cm
	\item[$\bullet$ $(C^1_{m+1})$:]
\[	[X_{ij}(s), X_{jk}(t)] = X_{ik}(st), \text{ for distinct } i, j, k, \text{where } \deg(st) \leq m+1. \]
	
	\vskip .2cm
	
	\noindent\textit{Indeed,} when $\deg(t)=0$,  $s=s_1s_2$ , $\deg(s_1)=m$,  $\deg(s_2)=1$, we have \begin{align*}
	&[X_{ij}(s_1s_2), X_{jk}(1)]   \\=&[[T_{kj}(1,s_2),X_{ij}(s_1)], X_{jk}(1)]\\     =&[-2X_{jk}(s_2), X_{ij}(s_1)]   +[T_{kj}(1,s_2), X_{ik}(s_1)]  \\ =&[-2X_{jk}(s_2), [X_{ik}(s_1),X_{kj}(1)]]   +[T_{kj}(1,s_2), X_{ik}(s_1)]  \\ \stackrel{\tiny\mbox{(\ref{C^3m})}}{=}& -2[ [X_{ik}(s_1),T_{jk}(s_2,1)]]   +[T_{kj}(1,s_2), X_{ik}(s_1)]\\   =&2[T_{jk}(s_2,1),X_{ik}(s_1)]   -[T_{jk}(s_2,1),X_{ik}(s_1)]   (=X_{ik}(s))   \\=&[X_{ij}(s_1), X_{jk}(s_2)] 
	\end{align*} (Note that we used frequently here Corollary \ref{any  ad works}).  This argument can also be applied to $[X_{ij}(1), X_{jk}(s_2s_1)]$,  using symmetry. Thus one can assume henceforth $\deg(t)\geq 1$.
	
	\vskip .2cm 
	Now we prove the general version. Now our goal is to prove:

\begin{claim}\label{newclaim2}
	{{Let $n=3$. For $w=\prod\limits_{i=1}^{m+1}w_i$ (e.g $\prod\limits_{1}^{m+1}w_i=w_1w_2$) , where $\deg(w)=m+1,\  \deg({w_i})=1 $, all expressions $[X_{ij}(\prod\limits_{l=1}^{p}w_l), X_{jk}(\prod\limits_{l=p+1}^{m+1}w_l)]$ are equal to $X_{ik}(w)$}. 
}
\end{claim}
	
	\vskip .2cm
	
	Note that when $p=1$ or $m+1$, this has  already been shown by what is done in the preceding paragraph. Without loss of generality, set $i=1, j=3, k=2$, and we're down to proving:
	\begin{claim}\label{n=3multid} Let $n=3$. Assuming $m\geq  \max\{|\mathcal{X}|+3, 10\}$, $\deg(abc)=m+1$ and $\deg(a), \deg(b), \deg(c) \geq 1$, the equality 
		\[[X_{13}(a) ,X_{32}(bc)] =[X_{13}(ab) ,X_{32}(c)],\] holds in $L_m$.
	\end{claim}
	
	\vskip .2cm
	
\noindent\textit{Proof of Claim \ref{n=3multid}.}	It follows from  Proposition \ref{left=right ad} that:
	
	\begin{itemize}
		\item[$(i)$] if $\deg(b) \geq 2$, then $[X_{13}(a) ,X_{32}(bc)]$ 
	$\stackrel{\tiny\mbox{(\ref{C^3m})}}{=}$
	$[T_{13}(a,1), X_{12}(bc)]$$=[ T_{32}(c,1),X_{12}(ab)]$ $\stackrel{\tiny\mbox{(\ref{C^3m})}}{=} [X_{13}(ab) ,X_{32}(c)]$. 
	
	\vskip .3cm 
	\item[$(ii)$] When $\deg(b)=1$, we deploy the following trick:  \begin{align*}
	& [X_{13}(a) ,X_{32}(bc)]   -[X_{13}(ab) ,X_{32}(c)] \\ 
	 \stackrel{\tiny\mbox{(\ref{C^3m})}}{=}  &
	 [X_{12}(a), T_{23}(1,bc)]  - [X_{12}(a), T_{23}(b,c)]  \\ \stackrel{\tiny\mbox{(\ref{C^3m})}}{=} &  [X_{12}(a), T_{23}(1,bc) - T_{23}(b,c)]. 
	\end{align*} 
	
	Now write $b=de$. The standard equalites regarding $T$ in $L_m$ (i.e. Lemma \ref{t eqn}) gives\begin{gather*}
	  T_{31}(cd,e)+T_{23}(de,c)\\ = T_{31}(cd,e)-T_{32}(c,de)=T_{21}(d,ec)\\ =T_{23}(1,dec)+T_{31}(d,ec),
	\end{gather*} which is  \[T_{23}(1,bc)-T_{23}(1,bc)=T_{31}(d,ec)-T_{31}(cd,e)\] after rearrangement. Therefore \begin{align*}
	 &\ [X_{12}(a), T_{23}(1,bc) - T_{23}(b,c)]   \\=&\ [X_{12}(a), T_{31}(d,ec)-T_{31}(cd,e)]\\    =&\ [X_{13}(ec),X_{32}(da)]-[X_{13}(e),X_{32}(cda)]. 
	\end{align*}
	
	But as $m$ is big enough, we can always choose $\deg(c) \geq 2$ (elsewise reverse the role of $c$ and $a$), so we're back to the $deg(b) \geq 2$ case proven in the previous paragraph. This proves Claim \ref{n=3multid}. \hfill  $\square$ \end{itemize}

	\vskip .2cm
	
Having settled Claim \ref{n=3multid} (and therefore Claim \ref{newclaim2}), the proof of ($C^1_{m+1}$) is complete.

	\vskip .2cm
	
		\item[$\bullet$ $(C^2_{m+1})$:]
\[	[X_{ij}(s), X_{ij}(t)] = 0, \text{ for } i\neq j, \  \deg(st)\leq m+2.\]
	
	\textit{Indeed,}  without loss of generality, assume $s=yx$, where $\deg(x)=1, \deg(y)=m,$ $\deg(t)=1$ . Take $k\neq i, j$. Then we have \begin{align*} &\ [X_{ij}(yx), X_{ij}(t)]\\
	=&\ [[X_{ik}(y),X_{kj}(x)],X_{ij}(t)]\\=&\ [[X_{ik}(y),X_{ij}(t)],X_{kj}(x)]   +[X_{ik}(y), [X_{kj}(x),X_{ij}(t)]]\\\stackrel{\mbox{\tiny ( \ref{C^3m})}}{=}&\ 0+0 = 0.\end{align*}  This proves $(  C^2_{m+1})$.
	
	\vskip .2cm
	
	\item[$\bullet$ $(C^3_{m+1})$:]
\[	[\boldsymbol{w},[X_{ij}(s), X_{kl}(t)]] = 0,  \text{ for } i=k \text{ or } j=l, \deg(st)\leq m+2, \boldsymbol{w}= X_{ij}(\gamma), \] {for an arbitrary letter }$\gamma \in \mathcal{X} \cup \{1\}$.
	
	\vskip .2cm
	
	\vskip .2cm
	\noindent\textit{Indeed,}\mbox{}  without loss of generality,  set $i=k=1, \ j=2, \ l=3$, and assume $\deg(s) \geq \deg(t)$.  {We may also assume $\deg(t) \geq 1$}, as $s=s_1s_2, \ \deg(s)=m+1,\ \deg(s_1)=1 $ implies   $[X_{ij}(s_1s_2), X_{ik}(1)]$$=[[T_{ik}(s_1,1), X_{ij}(s_2)], X_{ik}(1)]$ $\stackrel{\mbox{\tiny(\ref{C^3m})}}{=}$$[[T_{ik}(s_1,1) , X_{ik}(1)], X_{ij}(s_2)]$  $=2[X_{ik}(s_1), X_{ij}(s_2)]$. 
	
	\vskip .2cm
	
	Adhering to the above assumptions,  one may fall back to an  argument similar to what was done in Lemma \ref{main trick}: express $s=s_3s_1s_2=z_1z_2$, where $\deg(z_1)=\deg(s_1)=\deg(s_2)=1$,$\deg(s_3)>0$. 
	It follows that  \begin{align*}
	 & \ [X_{12}(s_3(s_1s_2-s_2s_1)),X_{13}(t)] \\
	 \stackrel{\mbox{\tiny ( \ref{C^3m} , \ref{C^1m})}}{=}  & \ \bigl[ [[T_{12}(1,s_1),T_{23}(1,s_2)],X_{12}(s_3)], X_{13}(t)\bigr] \\ \stackrel{\mbox{\tiny(\ref{C^3m} , \ref{C^1m}, Jacobi id)}}{=}  & \ 0+0=0. 
	\end{align*}Note also a symmetry: \[ [X_{12}(s),X_{13}(t)]=[X_{12}(s),[X_{12}(t),X_{23}(1)]] 
	 \stackrel{(\tiny C^2_{m+1})}{=}   [X_{12}(t),X_{13}(s)]. \] 
	  Finally, for $2\leq deg(s), \deg(t) \leq m$,  $deg(st) \leq m+2$, one may further reallocate degrees between arguments (using  $X_{13}(t)=[X_{12}(t_1),X_{23}(t_2)]$, where $t=t_1t_2$, $\deg(t_2)=1$):   	  
	  \[ [X_{12}(s),X_{13}(t)]=[X_{12}(s),[X_{12}(t_1),X_{23}(t_2)]] 
	  \stackrel{(\tiny C^1_{m+1}, \ C^2_{m+1})}{=}   [X_{12}(t_1),X_{13}(st_2)]. \]

	\vskip .2cm 
	
	The above may be summarized as:
	
	\begin{claim}\label{lastCclaim}
In $L_m$, let $\deg(st)=m+2$, $s=s_1s_2$, $t=t_1t_2$, where $\deg(s_1)=\deg(t_1)=1$. Then \begin{enumerate}
	\item[($a$)] \[ [X_{12}(s),X_{13}(t)]=[X_{12}(s'),X_{13}(t')], \] where $s'=s_1s_2'$ and $t'=t_1t_2'$, $s_2'$ (resp. $t_2'$) is any permutation of $s_2$ (resp. $t_2$); under the same assumption, all $[X_{13}(s_2's_1),X_{23}(t_2't_1)]$ are equal. The same equalites remain valid after $\{1,2,3\}$, appearing as subscripts of $X$, are permuted.
	\item[($b$)] \[[X_{ij}(t),X_{kl}(s)]=[X_{ij}(s),X_{kl}(t)].\]
	\item[($c$)]
	The following equalites remain valid after any permutation of indices:
\begin{align*}
		 [X_{12}(s_1s_2),X_{13}(t_1t_2)] &=[X_{12}(s_1s_2t_2),X_{13}(t_1)],\\
		 [X_{13}(s_1s_2),X_{23}(t_1t_2)]&=[X_{13}(t_2),X_{23}(t_1s_1s_2)].
\end{align*}

\end{enumerate}
	\end{claim}
 
	 	 We come back to validating  $(C^3_{m+1})$.	 Our analysis splits into cases regarding $\boldsymbol{w}$, in a way similar to $(B^3_{m+1})$.
	 	 
	 \begin{enumerate}
	 	\item If $\boldsymbol{w}  \in {X}_{12} \cup {X}_{13}$, then the desired relation follows from 
	 	(\ref{C^3m}) and (\ref{C^2m}) by a direct application of the Jacobi identity;
	 	
	 	\item If $\boldsymbol{w} \in $ ${X}_{32} \cup {X}_{23}$, then without loss of generality one may assume that $\boldsymbol{w}=X_{23}(c)$, where $c\in R$ and $\deg(c)\leq 1$. Assume additionally $\deg(a_i), \deg(b_i) \geq 3$. Then: \begin{align*}
	 	&[{X}_{23}(c), [{X}_{12}(a_1b_1), {X}_{13}(a_2b_2)]]\\
	 	=& [{X}_{23}(c), [[{X}_{13}(a_1), {X}_{32}(b_1)], [{X}_{12}(a_2),{X}_{23}(b_2)]]\\
	 	\stackrel{\mbox{\tiny (By \ref{C^2m} ,  \ref{C^3m})}}{=}& [{X}_{23}(c), [{X}_{13}(a_1),[{X}_{12}(a_2), [{X}_{32}(b_1), {X}_{23}(b_2)] ]]   \\
	 		\stackrel{\mbox{\tiny (By  \ref{C^3m})}}{=}& \ 0 - [{X}_{13}(a_1),[{X}_{13}(a_2c), {T}_{32}(b_1,b_2)] ]\\
	 	& +[{X}_{13}(a_1),[{X}_{12}(a_2), [X_{23}(c),{T}_{32}(b_1,b_2)] ]]\\
	 	=& \ 0 - [{X}_{13}(a_1),[{X}_{13}(a_2c), {T}_{32}(b_1,b_2)] ]\\
	 	& +[{X}_{13}(a_1),[{X}_{12}(a_2), X_{23}(b_2b_1c+cb_1b_2)]]\\
	 	\stackrel{\mbox{\tiny (By  \ref{C^3m})}}{=}& - [{X}_{13}(a_1),[{X}_{13}(a_2c), {T}_{32}(b_1,b_2)] ]+0 \\	 
	 	\stackrel{\mbox{\tiny (By  \ref{C^3m})}}{=}& - [{X}_{13}(a_1),{X}_{13}(a_2cb_1b_2)] = [{X}_{13}(a_1),[{X}_{12}(a_2c),X_{23}(b_1b_2)]]\\ \stackrel{\mbox{\tiny (By  \ref{C^3m})}}{=}& 0. 	
	 	\end{align*}
	 	
	 	\item  If $\boldsymbol{w}  \in {X}_{21} \cup {X}_{31}$, it suffices to consider the latter. As $X_{31}(c)=[X_{32}(c),X_{21}(1)]$, we work only with $\boldsymbol{w}=X_{21}(1)$. Assume again $\deg(a_i), \deg(b_i) \geq 3$. Then:
	 	
	 	\begin{align*}
	 	&[[{X}_{12}(a_1b_1), {X}_{13}(a_2b_2)],{X}_{21}(1)]\\
	 	=& [[[{X}_{13}(a_1), {X}_{32}(b_1)], [{X}_{12}(a_2),{X}_{23}(b_2)]],{X}_{21}(1)]\\
	 	\mbox{\scriptsize (By $C^3_m$) } =& [0+[{X}_{12}(a_2),[{X}_{13}(a_1), T_{32}(b_1,b_2)]], {X}_{21}(1)]    \\
	 	\mbox{\scriptsize (By $C^3_m$) } =& { [T_{12}(a_2,1),[{X}_{13}(a_1), {T}_{32}(b_1,b_2)] ]}_{} {-[{X}_{12}(a_2),[X_{23}(a_1), {T}_{32}(b_1,b_2)] ]}_{}\\		
	 	& {-[{X}_{12}(a_2),[{X}_{13}(a_1), {X}_{21}(b_2b_1)] ]}_{}\\
	\mbox{\scriptsize (By $C^3_m$) } =& { [T_{12}(a_2,1),[{X}_{13}(a_1), {T}_{32}(b_1,b_2)] ]}_{} {-[{X}_{12}(a_2),X_{23}(b_2 b_1 a_1+a_1 b_1 b_2)]}_{}\\		
	 	&{-[{X}_{12}(a_2),X_{23}(-b_2 b_1 a_1)]}_{}\\	
	 	=& [T_{12}(a_2,1),[X_{13}(a_1),T_{32}(b_1,b_2)]]-[X_{12}(a_2),X_{23}(a_1b_1b_2)]\\
	 	=&[X_{13}(a_2a_1),T_{32}(b_1,b_2)]+[X_{13}(a_1),[T_{12}(a_2,1),T_{32}(b_1,b_2)]]\\
	 	&-[X_{12}(a_2),X_{23}(a_1b_1b_2)]\\
	 	=&[X_{12}(-a_2b_2b_1),X_{23}(a_1)]+[X_{12}(a_2),X_{23}(b_2b_1a_1+a_1b_1b_2)]\\&+[X_{13}(a_1),[T_{12}(a_2,1),T_{32}(b_1,b_2)]]-[X_{12}(a_2),X_{23}(a_1b_1b_2)]\\
	 		 	=&\underbrace{[X_{12}(-a_2b_2b_1),X_{23}(a_1)]+[X_{12}(a_2),X_{23}(b_2b_1a_1)]}_{A}\\&+\underbrace{[X_{13}(a_1),[T_{12}(a_2,1),T_{32}(b_1,b_2)]]}_{B}
	 	\end{align*}
	 	
	 	Let us first analyze item $A$, namely \[-[[T_{13}(1,a_2),X_{12}(b_2b_1)],X_{23}(a_1)]+[X_{12}(a_2),[T_{21}(b_2,b_1),X_{23}(a_1)]].\] Set $a_1=a_1'a_1''$, $b_2=b_2'b_2''$ and $b_1=b_1''b_1'$, where $\deg(a_i'')$, $\deg(a_i') $,  $\deg(b_i'')$, $\deg(b_i') >0$. Then \begin{align*}
	 	&[X_{12}(a_2),X_{23}(b_2b_1a_1)]
	 	=[X_{12}(a_2),[X_{21}(b_2'),X_{13}(b_2''b_1a_1)]]\\
	 	\stackrel{\mbox{\tiny (\ref{C^3m})}}{=}& [[X_{12}(a_2),X_{21}(b_2')], X_{13}(b_2''b_1a_1)]
	 	 =[T_{12}(a_2,b_2'),[T_{23}(a_1,b_1'), X_{13}(b_2''b_1'')]].
	 	\end{align*} Similar analysis gives \begin{align*}
	 	&[X_{12}(a_2b_2b_1), X_{23}(a_1)]=[[X_{13}(a_2''b_2b_1''), X_{32}(b_1')],X_{23}(a_1)] \\  \stackrel{\mbox{\tiny (\ref{C^3m})}}{=}& 
	   [X_{13}(a_2b_2b_1''),[X_{32}(b_1'),X_{23}(a_1)]]
	   =[[T_{12}(a_2,b_2'),X_{13}(b_2''b_1'')],T_{32}(b_1',a_1)].
	 	\end{align*} 
	 	
	 	It follows that the item 
\[	 	A=[[T_{12}(a_2,b_2'),T_{23}(a_1,b_1')], X_{13}(b_2''b_1'')].\]

 The expression $A$ can be rewritten as \[[X_{12}(b_2''[b_2'a_2,a_1b_1']),X_{23}(b_1'')]-[X_{12}(b_2''), X_{23}([b_2'a_2,a_1b_1']b_1'')],\] while $B$ is equal to \[[X_{12}(a_1'[b_2b_1,a_2]),X_{23}(a_1'')]-[X_{12}(a_1'),X_{23}([b_2b_1,a_2]a_1'')]. \]

 We claim that $A=B=0$. First, we prove each of the following equalities:
\vskip .2cm
 \begin{itemize}
	\item[$\bullet$]  $[X_{13}(a_1'[b_2b_1,a_2]), X_{23}(a_1'')]=0$
	\item[$\bullet$]  $[X_{13}(a_1'), X_{23}([b_2b_1,a_2]a_1'')]=0$	\item[$\bullet$]  $[X_{13}(b_2''[b_2'a_2,a_1b_1']), X_{23}(b_1'')]=0$	\item[$\bullet$]  $[X_{13}(b_2''), X_{23}([b_2'a_2,a_1b_1']b_1'')]=0$
\end{itemize}
\vskip .2cm 
It follows from part ($a$) of Claim \ref{lastCclaim} that the second and fourth formulas are valid. For the first and third, note that the terminal words of the two summands of  $a_1'[b_2b_1,a_2]$ are $b_1$ and $a_2$, while the terminal words of the two summands of 
 
$b_2''[b_2'a_2,a_1b_1']$ are $a_2$ and $b_1'$.
 
As $b_1'$ is a terminal word of $b_1$, if the terminal letters of    $a_2$ and $b_1$ are   equal, deferring to ($a$) of Claim \ref{lastCclaim} allows all bulleted formulas to be verified.
\vskip.2cm
To obtain equality of these two terminal letters, invoke the condition $m\geq  (|\mathcal{X}|+3)$. By parts ($c$) and ($b$) of Claim \ref{lastCclaim},  we have \[[{X}_{12}(a_1b_1), {X}_{13}(a_2b_2)]=[{X}_{12}(a_2b_2b_1),{X}_{13}(a_1)].\] Part ($b$) of Claim \ref{lastCclaim} further allows us to assume $\deg(a_2b_2b_1) >  |\mathcal{X}|+2$. Writing $a_2=a_2'a_2''$ where $\deg(a_2')=1$, we see that  in the word $a_2''b_2b_1$ at least one letter from the alphabet $X$ appears twice or more. Using ($a$) of Claim \ref{lastCclaim}, we are allowed to rearrange $a_2''b_2b_1$ so  that in the new $a_2''$, $b_1$ (and of course, $b_1'$) the terminal letters are equal, while retaining the degrees of $a_2''$,  $b_2$,  and $b_1$. This proves all four bullet equalities.

\vskip.2cm
What remains to be done is now straightforward. Let us take $A$ for example. By the four bullet equalities above,  \begin{align*}
A &=[X_{12}(b_2''[b_2'a_2,a_1b_1']),X_{23}(b_1'')]-[X_{12}(b_2''), X_{23}([b_2'a_2,a_1b_1']b_1'')]\\
&=[X_{13}(b_2''[b_2'a_2,a_1b_1']),T_{32}(1,b_1'')]-[T_{12}(b_2'',1), X_{13}([b_2'a_2,a_1b_1']b_1'')]\\
&=\bigl[[T_{12}(1,b_2''),X_{13}([b_2'a_2,a_1b_1'])],T_{32}(1,b_1'')\bigr]-[T_{12}(b_2'',1),X_{13}([b_2'a_2,a_1b_1']b_1'')]\\
&=-\bigl[[T_{12}(1,b_2''),T_{23}(1,b_1'')], X_{13}([b_2'a_2,a_1b_1'])\bigr]. 
\end{align*}	 	

As $\deg(b_i') \geq 1$, $\deg(a_i) \geq 3$, we may apply the same argument as was used in Lemma \ref{main trick} (recall that degree $m+1$ elements are defined and behave as expected, thanks to the proven relation $(C_{m+1}^1)$). It follows that $A=0$.

\vskip.2cm
The expression $B$ is treated no differently: similar computations yield
\[ B=-\bigl[[T_{12}(1,a_1'),T_{23}(1,a_1'')], X_{13}([b_2b_1,a_2])\bigr],\] and the same Lemma \ref{main trick} type argument gives $B=0$. 

Since $[[{X}_{12}(a_1b_1), {X}_{13}(a_2b_2)],{X}_{21}(1)]=A+B=0$, the relation $(C_{m+1}^3)$ is proven.

	\end{enumerate}
	 
	\vskip .2cm
	This proves Lemma \ref{relm+1}, therefore Proposition \ref{claim2}, Lemma \ref{mainlemma} and (finally) Theorem \ref{main}. \qedhere
\end{itemize}

\bibliographystyle{plain}  
\bibliography{FPShenzhenmergedBib}   

 \end{document}